\documentclass[a4paper]{amsart}

\usepackage{a4wide}
\usepackage[T1]{fontenc}
\usepackage[utf8]{inputenc}
\usepackage{lmodern}
\usepackage{microtype}
\usepackage{amsmath,amssymb,mathtools}
\usepackage{xcolor}
\usepackage{url}
\usepackage[
  colorlinks=true,
  linkcolor=blue!55!black,
  citecolor=blue!55!black,
  urlcolor=blue,
  pdftitle={Pointwise function spaces over compacta are not weak Banach spaces},
  pdfauthor={Tomasz Kania and Jerzy K\k{a}kol}
]{hyperref}

\newtheorem{theorem}{Theorem}[section]
\newtheorem{proposition}[theorem]{Proposition}
\newtheorem{lemma}[theorem]{Lemma}
\newtheorem*{theoremA}{Theorem A}
\newtheorem*{corollaryB}{Corollary B}

\newcommand{\R}{\mathbb{R}}
\newcommand{\N}{\mathbb{N}}
\newcommand{\abs}[1]{\lvert #1\rvert}
\newcommand{\norm}[1]{\lVert #1\rVert}
\newcommand{\Fin}{\mathcal F}

\title[Pointwise and weak topologies]{Pointwise function spaces over compacta\\
are not weak Banach spaces}

\author[T.~Kania]{Tomasz Kania}
\address[T.~Kania]{Mathematical Institute\\Czech Academy of Sciences\\
\v Zitn\'a 25 \\115 67 Praha 1\\Czech Republic  and  Institute of
Mathematics and Computer Science\\ Jagiellonian University\\
{\L}ojasiewicza 6, 30-348 Krak\'{o}w, Poland
}
\email{kania@math.cas.cz, tomasz.marcin.kania@gmail.com}
\thanks{RVO: 67985840.}

\author[J.~K\k{a}kol]{Jerzy K\k{a}kol}
\address[J.~K\k{a}kol]{Faculty of Mathematics and Computer Science\newline
Adam Mickiewicz University\newline
Uniwersytetu Pozna\'nskiego 4, 61-614 Pozna\'n, Poland}
\email{kakol@amu.edu.pl}


\begin{document}

\begin{abstract}
Let $K$ be a compact Hausdorff space and let $E$ be an
infinite-dimensional real Banach space.  We prove that there is no
continuous bijection $h\colon C_p(K)\to E_w$ whose inverse is continuous
at $h(0)$.  Consequently, $C_p(K)$ and $C_w(L)$ are not homeomorphic for
any infinite compact Hausdorff spaces $K$ and $L$.  This settles
Krupski's problem and its two-space version due to Krupski and
Marciszewski, and answers a question of K\k{a}kol, Leiderman, and Michalak
concerning $C_p([0,1])$ and weak Banach spaces.
\end{abstract}

\maketitle

\section{Introduction}

For a Tychonoff space $X$, let $C_p(X)$ denote the space of continuous
real-valued functions on $X$ endowed with the topology of pointwise
convergence.  If $E$ is a real Banach space, $E_w$ denotes $E$ endowed
with its weak topology.  For compact Hausdorff $K$, we write $C_w(K)$
for $C(K)_w$, where $C(K)$ carries its usual supremum norm before the
weak topology is imposed.  The notation $X\cong Y$ means that the
topological spaces $X$ and $Y$ are homeomorphic; no linearity is implied.

The topological classification of spaces $C_p(X)$ is a central part of
the programme initiated by Arkhangel'ski\u{\i}; see, for example,
\cite{Arkhangelskii1992}.  The precise weak-versus-pointwise question
considered here appears in the published literature as a problem of
Krupski.  In \cite[Problem~2]{Krupski2016} he asked whether
\begin{equation}\label{eq:one-space-problem}
                         C_p(K)\cong C_w(K)
\end{equation}
can hold for an infinite compact space $K$.  Krupski and Marciszewski
then isolated the two-space form \cite[Problem~1.2]{KrupskiMarciszewski2017}:
can $C_p(K)$ and $C_w(L)$ be homeomorphic for infinite compact spaces
$K$ and $L$?  Although this comparison belongs naturally to
Arkhangel'ski\u{\i}'s wider programme, we know no published source that
attributes the precise problem \eqref{eq:one-space-problem} to him.

Several substantial negative results were previously known.  Krupski
settled the one-space problem for compact metrisable $C$-spaces, and
hence for compact finite-dimensional metrisable spaces
\cite[Corollary~1]{Krupski2016}.  Krupski and Marciszewski treated, among
other cases, finite-dimensional Valdivia compacta
\cite[Corollary~4.5]{KrupskiMarciszewski2017}.  They also observed that, when
$K$ is scattered, the two-space case follows from the
Fr\'{e}chet--Urysohn property of $C_p(K)$ and its failure for $C_w(L)$
when $L$ is infinite
\cite[item~(D), p.~648]{KrupskiMarciszewski2017}.  Uniform mapping
problems were developed by G\'{o}rak, Krupski, and Marciszewski
\cite{GorakKrupskiMarciszewski2019}.  In a broader Banach-space
direction, K\k{a}kol, Leiderman, and Michalak proved that a homeomorphism
$C_p(X)\cong E_w$ forces $X$ to be a countable union of compacta, at
least one of them non-scattered, and forces $E$ to contain an isomorphic
copy of $\ell_1$ \cite{KakolLeidermanMichalak2022}.  They asked in
\cite[Problem~3.10]{KakolLeidermanMichalak2022} whether $C_p([0,1])$ can be
homeomorphic to $E_w$ for some separable Banach space $E$.

Our main result removes all additional assumptions when the source is
defined over a compact space.  It also weakens the homeomorphism
hypothesis.

\begin{theoremA}
Let $K$ be a compact Hausdorff space and let $E$ be an
infinite-dimensional real Banach space.  There is no continuous
bijection
\[
                              h\colon C_p(K)\longrightarrow E_w
\]
whose inverse is continuous at $h(0)$.
\end{theoremA}

Taking $E=C(L)$ gives the answer to the problem that motivated the
paper.

\begin{corollaryB}
If $K$ and $L$ are infinite compact Hausdorff spaces, then
\[
                              C_p(K)\not\cong C_w(L).
\]
In particular, $C_p(K)$ and $C_w(K)$ are not homeomorphic.
\end{corollaryB}

One cannot replace bijectivity in Theorem~A by surjectivity.  Indeed, if
$S$ is a convergent sequence and $X$ is an infinite compact metrisable
space, the construction of Krupski and Marciszewski
\cite[Proposition~5.4]{KrupskiMarciszewski2016} gives a continuous
surjection from $C_p(S)$ onto $C(X)$, even for the norm topology on
$C(X)$.  This does not show that continuity of $h^{-1}$ at $h(0)$ can
be omitted while bijectivity is retained.

The proof combines a norm-compact linearly independent set in $E^*$ with
the finite-support method of Okunev, Krupski, and Marciszewski
\cite{Okunev2011,Krupski2013,Krupski2016,KrupskiMarciszewski2017}.

\section{Preliminaries}
\label{sec:preliminaries}

Let $\N=\{1,2,\ldots\}$.  The following lemma is well known.  More
generally, Banakh and Plichko proved that every infinite-dimensional
complete linear metric space contains a linearly independent copy of
every compact metric space \cite[Corollary~1]{BanakhPlichko2006}.  We
include the proof for completeness; the explicit norm parametrisation
will also enter the proof of Theorem~A.

\begin{lemma}\label{lem:arc}
Let $E$ be an infinite-dimensional real Banach space.  There is a
norm-continuous map $[0,1]\to B_{E^*}$, $t\mapsto\mu_t$,
whose image $P$ is homeomorphic to $[0,1]$, does not contain $0$, and is
linearly independent.  Consequently, if $V\subseteq E^*$ is
linear and $\dim V=r<\infty$, then $\abs{V\cap P}\leqslant r$.  The norm
and weak-star topologies coincide on $P$.
\end{lemma}

\begin{proof}
The dual $E^*$ is infinite-dimensional.  By Mazur's basic-sequence
selection theorem, it contains a normalised basic sequence
$(\nu_j)_{j\geqslant0}$; see
\cite[Section~1.5]{AlbiacKalton2016}.  We use the convention $t^0=1$
for every $t\in[0,1]$, including $t=0$.  For $t\in[0,1]$, let
$\mu_t=\sum_{j=0}^{\infty}2^{-j-1}t^j\nu_j$.
The series converges absolutely in norm and $\norm{\mu_t}\leqslant1$.
Moreover, the terms corresponding to $j=0$ cancel in
$\mu_t-\mu_s$, and the mean value theorem gives
$\abs{t^j-s^j}\leqslant j\abs{t-s}$ for $j\geqslant1$ and
$s,t\in[0,1]$.  Therefore
$\norm{\mu_t-\mu_s}\leqslant
\sum_{j=1}^{\infty}2^{-j-1}\abs{t^j-s^j}\leqslant
\abs{t-s}\sum_{j=1}^{\infty}j2^{-j-1}=\abs{t-s}$.  Thus the map is
norm-continuous.

Let us suppose that $t_1,\ldots,t_q$ are distinct and
$a_1\mu_{t_1}+\cdots+a_q\mu_{t_q}=0$.
Uniqueness of the coefficients with respect to the basic sequence gives
$a_1t_1^j+\cdots+a_qt_q^j=0$ for every $j\geqslant0$.  For
$j=0,\ldots,q-1$, the coefficient matrix is the transpose of a
Vandermonde matrix.  Since the $t_i$ are distinct, it is nonsingular by
the standard determinant formula
\cite[Section~VM, Theorem~DVM]{Beezer2004}, so
$a_1=\cdots=a_q=0$.  Hence $P=\{\mu_t:0\leqslant t\leqslant1\}$ is
linearly independent.  In particular the map is injective, and
the non-zero coefficient of $\nu_0$ shows that $0\notin P$.  A continuous
injection from the compact interval into the normed space $E^*$ is a
homeomorphic embedding, so $P$ is a norm-compact arc and is perfect.

Any subset of an $r$-dimensional space having more than $r$ elements is
linearly dependent, which proves the stated estimate.  Finally,
the identity from $P$ with its norm topology to $P$ with its weak-star
topology is a continuous bijection from a compact space to a Hausdorff
space.  It is therefore a homeomorphism.
\end{proof}

The next lemma is the countable $k$-uniform instance of the
$\Delta$-system theorem of Erd\H{o}s and Rado
\cite[p.~86]{ErdosRado1960}.  We include its short proof in the form
used below.

\begin{lemma}\label{lem:delta-system}
Let $k\in\N$ and let $(F_n)$ be a sequence of pairwise distinct
$k$-element sets.  There are strictly increasing indices
$n_1<n_2<\cdots$, a finite set $A$, and pairwise disjoint non-empty
finite sets $A_j$, disjoint from $A$, such that
$F_{n_j}=A\cup A_j$ for every $j\in\N$.  In particular, $\abs A<k$.
The set $A$ is called the root of this $\Delta$-system, and the sets
$A_j=F_{n_j}\setminus A$ are called its petals.
\end{lemma}

\begin{proof}
We argue by induction on $k$.  The case $k=1$ is immediate, with
$A=\varnothing$.  Let us suppose that $k>1$.  If some point $x$ belongs to
infinitely many $F_n$, let us retain those sets and apply the induction
hypothesis to the pairwise distinct sets $F_n\setminus\{x\}$; then
let us adjoin $x$ to the root.  If no point belongs to infinitely many $F_n$,
let us choose a pairwise disjoint subsequence recursively.  Indeed, a finite
union of previously chosen sets meets only finitely many $F_n$.  In the
latter case, let us take $A=\varnothing$.
\end{proof}

\section{Proof of the main theorem}
\label{sec:main-proof}

Let us suppose, towards a contradiction, that a map as in Theorem~A exists.
Translation in $E_w$ allows us to assume that
$h\colon C_p(K)\to E_w$ satisfies $h(0)=0$,
while preserving continuity of $h$ and continuity of $h^{-1}$ at the
origin.  If $K=\varnothing$, surjectivity is already impossible, so let
us assume that $K$ is non-empty.  Let us fix the arc
$P\subseteq B_{E^*}$ supplied
by Lemma~\ref{lem:arc}.
In what follows, $P$ carries its norm topology, which coincides with
its weak-star topology by Lemma~\ref{lem:arc}.

For $k\in\N$, let
$\Fin_k(K)=\{F\subseteq K:1\leqslant\abs F\leqslant k\}$ carry the
Vietoris topology, and let $\Fin_0(K)=\varnothing$.  The map from $K^k$
onto $\Fin_k(K)$ which sends $(x_1,\ldots,x_k)$ to
$\{x_1,\ldots,x_k\}$ is continuous: the inverse image of a basic
Vietoris set $\langle U_1,\ldots,U_s\rangle$ consists of those tuples
whose coordinates lie in $\bigcup_{i=1}^s U_i$ and meet every $U_i$,
and is therefore open.  Hence $\Fin_k(K)$ is compact.  Since $K$ is
Hausdorff, its Vietoris hyperspace, and therefore $\Fin_k(K)$, is
Hausdorff.  For a
finite $F\subseteq K$ and $m\in\N$, let
$U(F,m)=\{f\in C(K):\abs{f(x)}<1/m\text{ for }x\in F\}$; in particular,
$U(\varnothing,m)=C_p(K)$.  For $\mu\in P$, let
$H(\mu)=\{y\in E:\abs{\mu(y)}\leqslant1\}$.

The following relation is adapted from the relation in the proof of
\cite[Theorem~4.1]{KrupskiMarciszewski2017}.  Let
$\pi_2\colon\Fin_k(K)\times P\to P$ denote the projection onto the
second coordinate, so that $\pi_2(F,\mu)=\mu$.  Let
\[
 \begin{split}
 Z_{k,m}&=\{(F,\mu)\in\Fin_k(K)\times P:
                    h(U(F,m))\subseteq H(\mu)\},\\
 D_{k,m}&=\pi_2(Z_{k,m}),\qquad D_{0,m}=\varnothing.
 \end{split}
\]
The fixed-support estimate below is the analogue of
\cite[Claim~1, p.~651]{KrupskiMarciszewski2017}.  The linearly independent
set $P$ makes the present fibre finite.

\begin{proposition}\label{prop:closed-cover}
For all $k,m\in\N$, the set $Z_{k,m}$ is compact and $D_{k,m}$ is a
compact, hence closed, subset of $P$.  Moreover,
$D_{k-1,m}\subseteq D_{k,m}$ and
$P=\bigcup_{k,m\in\N}D_{k,m}$.
\end{proposition}

\begin{proof}
If $(F,\mu)\notin Z_{k,m}$, let us choose $f\in U(F,m)$ such that
$\abs{\mu(h(f))}>1$.
For every $x\in F$, continuity of $f$ gives an open neighbourhood
$V_x$ of $x$ on which $\abs f<1/m$.  Let
$V=\bigcup_{x\in F}V_x$.  Hence $f\in U(F',m)$ whenever
$F'\in\Fin_k(K)$ and $F'\subseteq V$.  The latter condition defines an
open Vietoris neighbourhood of $F$.  At the same time,
$\nu\mapsto\nu(h(f))$ is norm-continuous on $P$, so the strict inequality
persists on a relatively norm-open neighbourhood of $\mu$ in $P$.
These two neighbourhoods form a product
neighbourhood disjoint from $Z_{k,m}$.  Thus $Z_{k,m}$ is closed in the
compact Hausdorff space $\Fin_k(K)\times P$, and is compact.  Its second
projection $D_{k,m}$ is compact and hence closed in $P$.

The asserted inclusion follows from
$\Fin_{k-1}(K)\subseteq\Fin_k(K)$.  For the equality, let us fix $\mu\in P$.
Since $h(0)=0$, we have $\mu(h(0))=0$.  As $h$ is continuous into
$E_w$, the map
$f\mapsto\mu(h(f))$ is continuous at $0$ in $C_p(K)$.  There are a
finite set $F\subseteq K$
and $\varepsilon>0$ such that $\abs{f(x)}<\varepsilon$ for every
$x\in F$ implies $\abs{\mu(h(f))}<1$.
The set $F$ cannot be empty: otherwise surjectivity of $h$ would imply
$\abs{\mu(y)}<1$ for every $y\in E$, contrary to $\mu\ne0$.  Let us choose
$m\in\N$ with $1/m<\varepsilon$ and let $k=\abs F$.  Then
$(F,\mu)\in Z_{k,m}$, so $\mu\in D_{k,m}$.
\end{proof}

\begin{lemma}\label{lem:fixed-support}
For every $k,m\in\N$ and every $F\in\Fin_k(K)$ there are
$r\in\N$ and $\varphi_1,\ldots,\varphi_r\in E^*$, depending only on $(F,m)$,
such that
\[
 \{\mu\in P:(F,\mu)\in Z_{k,m}\}
 \subseteq P\cap\operatorname{span}\{\varphi_1,\ldots,\varphi_r\},
 \qquad
 \abs{\{\mu\in P:(F,\mu)\in Z_{k,m}\}}\leqslant r.
\]
\end{lemma}

\begin{proof}
Continuity of $h^{-1}$ at $0$ gives a weak neighbourhood of $0$ in $E$
of the form $W=\{y\in E:\abs{\varphi_i(y)}<\eta,
1\leqslant i\leqslant r\}$ such that
$h^{-1}(W)\subseteq U(F,m)$.  Since $h$ is bijective,
$W\subseteq h(U(F,m))$.

Let us fix $\mu\in P$ with $(F,\mu)\in Z_{k,m}$.  If
$\mu\notin\operatorname{span}\{\varphi_1,\ldots,\varphi_r\}$, there is
$y\in E$ such that
$\varphi_i(y)=0$ for $1\leqslant i\leqslant r$ and $\mu(y)=2$.
Indeed, for $T\colon E\to\R^r$, given by
$T(y)=(\varphi_1(y),\ldots,\varphi_r(y))$, if
$\ker T\subseteq\ker\mu$, the formula $\lambda(Ty)=\mu(y)$ defines a
linear functional on $T(E)$.  Since $T(E)$ is finite-dimensional,
$\lambda$ extends to a linear functional on $\R^r$.  Consequently,
$\mu=\lambda\circ T\in\operatorname{span}\{\varphi_1,\ldots,\varphi_r\}$.
Thus our assumption gives $y_0\in\ker T$ with $\mu(y_0)\ne0$; replacing
$y_0$ by $2y_0/\mu(y_0)$ gives the required vector $y$.

$y\in W\subseteq h(U(F,m))$, whereas
$h(U(F,m))\subseteq H(\mu)$.  This contradicts $\mu(y)=2$ and proves
the asserted inclusion.  The cardinal estimate follows from
Lemma~\ref{lem:arc}.
\end{proof}

We next recall a standard lemma, which is a consequence of the
observation on point-finite open expansions recorded in
\cite[p.~2796]{Sakai2006}.  We include a direct proof for compact spaces
for completeness.

\begin{lemma}\label{lem:cutoff}
Let $K$ be a compact Hausdorff space and let $(A_n)$ be a sequence of
pairwise disjoint, non-empty finite subsets of $K$ such that
$\abs{A_n}\leqslant q$ for some $q\in\N$.  There are strictly increasing
indices $n_1<n_2<\cdots$ and functions $u_j\in C(K,[0,1])$ such that
$u_j=1$ on $A_{n_j}$ and $u_j(x)\to0$ for every $x\in K$.
\end{lemma}

\begin{proof}
We argue by induction on $q$.  Let us first suppose that each $A_n$ is a
singleton, say $A_n=\{x_n\}$.  The points $x_n$ are distinct.  Compactness
gives an $\omega$-accumulation point $x$ of the set $\{x_n:n\in\N\}$:
every neighbourhood of $x$ contains infinitely many $x_n$.
Indeed, otherwise every point of $K$ would have an open neighbourhood
meeting $\{x_n:n\in\N\}$ in only finitely many points.  A finite
subcover would then imply that this set is finite.

Starting with $W_1=K$, let us recursively choose strictly increasing $n_j$,
open neighbourhoods
$W_j$ of $x$, and pairwise disjoint open sets $V_j$ such that
$x_{n_j}\in V_j\subseteq W_j$ and
$W_{j+1}\subseteq W_j\setminus\overline{V_j}$.
At stage $j$, let us choose $n_j>n_{j-1}$ with
$x_{n_j}\in W_j\setminus\{x\}$.  Regularity gives an open $V_j$ with
$x_{n_j}\in V_j$ and
$\overline{V_j}\subseteq W_j\setminus\{x\}$.  Since
$x\notin\overline{V_j}$, an open neighbourhood $W_{j+1}$ with the
required inclusion exists.  Future sets $V_i$ lie in
$W_{j+1}$, so the $V_j$ are pairwise disjoint.

Let us choose an open set $O_j$ with $x_{n_j}\in O_j$ and
$\overline{O_j}\subseteq V_j$.  Urysohn's lemma gives
$v_j\in C(K,[0,1])$ such that $v_j(x_{n_j})=1$ and $v_j=0$ on
$K\setminus O_j$.  Thus
$\operatorname{supp}v_j\subseteq\overline{O_j}\subseteq V_j$.
The supports are pairwise disjoint, and hence $v_j\to0$ pointwise.  This
proves the case $q=1$.

Let us assume the result for $q-1$.
For each $n$, let us choose $x_n\in A_n$.  After applying the singleton
case and passing to a subsequence, let us choose pointwise-null
$v_n\in C(K,[0,1])$ with
$v_n(x_n)=1$.  Let us put $B_n=A_n\setminus\{x_n\}$.  The sets $B_n$ are
pairwise disjoint and have at most $q-1$ points.  If infinitely many of
them are empty, let us pass to those indices and put $w_n=0$.  Otherwise,
let us discard the finitely many empty sets, apply the induction
hypothesis, and pass to a further subsequence to obtain pointwise-null
$w_n\in C(K,[0,1])$ with $w_n=1$ on $B_n$.  Then
$u_n=\max\{v_n,w_n\}$ is continuous and equals $1$ on $A_n$.  Since
$0\leqslant u_n\leqslant v_n+w_n$, the sequence $(u_n)$ is pointwise
null.
Relabelling the retained indices proves the assertion.
\end{proof}

The cardinality bound cannot be omitted.  Indeed, without it the
conclusion of the lemma would give, for every pairwise disjoint sequence
of non-empty finite subsets of $K$, a subsequence $(A_{n_j})$ and
functions $(u_j)$ as above.  The sets
$\{x\in K:u_j(x)>1/2\}$ would then form a point-finite open expansion of
that subsequence.  This is precisely property $(\kappa)$.  For compact
$K$, \cite[Theorem~5.9]{KrupskiMarciszewski2017} and
\cite[Theorem~1.2]{KrupskiKucharskiMarciszewski2025} show that $K$ has
property $(\kappa)$ if and only if it is scattered.  Thus the conclusion
without the cardinality bound fails for every non-scattered compact
space.

In the proof below, compactness of $K$ ensures that each $\Fin_k(K)$ is
compact and that the sets $D_{k,m}$ are closed in $P$.  This permits a
direct application of the Baire category theorem to the resulting
closed cover of $P$.  The conclusion of Lemma~\ref{lem:cutoff} holds for
arbitrary Tychonoff spaces: the observation cited above provides a
point-finite open expansion of a subsequence, and complete regularity
supplies the corresponding continuous functions.  The subsequent
resolution of the broader question raised by K\k{a}kol, Leiderman, and
Michalak \cite{KakolLeidermanMichalak2022}, concerning homeomorphisms
$C_p(X)\cong E_w$ for arbitrary infinite Tychonoff spaces $X$ and Banach
spaces $E$, is described in the note added in proof.

\begin{proof}[Proof of Theorem~A]
Let us retain the notation and the contradiction hypothesis fixed above.  By
Proposition~\ref{prop:closed-cover}, the compact metric arc $P$ is the countable
union of the closed sets $D_{k,m}$.  The Baire category theorem gives
$k,m\in\N$ for which $D_{k,m}$ has non-empty interior in $P$.  Let us fix
such an $m$ and let $k$ be the least positive integer for which this
happens.  Then $\operatorname{int}_P D_{k-1,m}=\varnothing$.  Since
$D_{k-1,m}$ is closed in $P$, the set
$G=\operatorname{int}_P D_{k,m}\setminus D_{k-1,m}$ is a non-empty
relatively open subset of $P$.

The set $P$ is a norm-compact arc, so $G$ has no isolated points.  Let us choose
$\mu_\infty\in G$ and pairwise distinct functionals
$\mu_n\in G\setminus\{\mu_\infty\}$ such that
$\norm{\mu_n-\mu_\infty}\to0$.
For each $n$, let us choose $F_n\in\Fin_k(K)$ with
$(F_n,\mu_n)\in Z_{k,m}$.
Since $\mu_n\notin D_{k-1,m}$, each $F_n$ has exactly $k$ points.
For each fixed $F\in\Fin_k(K)$, Lemma~\ref{lem:fixed-support} makes the
fibre $\{\mu\in P:(F,\mu)\in Z_{k,m}\}$ finite.  Since the $\mu_n$ are
distinct, each fixed $F$ occurs only finitely often among the $F_n$.
Thus $\{F_n:n\in\N\}$ is infinite; if $K$ is finite this is already a
contradiction.  Let us pass to a subsequence and assume that the $F_n$
are pairwise distinct.

Let us apply Lemma~\ref{lem:delta-system} and relabel the resulting subsequence
so that $F_n=A\cup A_n$, where $A$ is the common root and the non-empty
sets $A_n=F_n\setminus A$ are the petals.  These sets are pairwise disjoint
and disjoint from $A$, and $\abs A<k$.

The minimal choice of $k$ gives
$h(U(A,m))\nsubseteq H(\mu_\infty)$.
Indeed, if $A\ne\varnothing$ and the reverse inclusion held, then
$\mu_\infty\in D_{\abs A,m}\subseteq D_{k-1,m}$, contrary to
$\mu_\infty\in G$.  If $A=\varnothing$, the reverse inclusion would say
$E=h(C_p(K))\subseteq H(\mu_\infty)$, which is impossible because
$\mu_\infty\ne0$.  Let us choose $f_0\in U(A,m)$ such that
$\abs{\mu_\infty(h(f_0))}>1$.
Since $\norm{\mu_n-\mu_\infty}\to0$,
$\mu_n(h(f_0))\to\mu_\infty(h(f_0))$.  There is therefore
$\varepsilon>0$ such that, after discarding finitely many terms,
\begin{equation}\label{eq:margin}
                         \abs{\mu_n(h(f_0))}>1+\varepsilon
                         \qquad(n\in\N).
\end{equation}

Let us apply Lemma~\ref{lem:cutoff} to the petals $A_n$, whose cardinalities
are at most $k$.  After another passage to a subsequence, let us choose
$u_n\in C(K,[0,1])$ such that $u_n=1$ on $A_n$ and $u_n\to0$
pointwise on $K$.  Let us put $g_n=(1-u_n)f_0$.
Then $g_n\to f_0$ in $C_p(K)$.  On $A_n$ we have $g_n=0$, whereas for
$x\in A$ we have
$\abs{g_n(x)}\leqslant\abs{f_0(x)}<1/m$.  Thus $g_n\in U(F_n,m)$.
Since $(F_n,\mu_n)\in Z_{k,m}$, we obtain
\begin{equation}\label{eq:upper-bound}
                              \abs{\mu_n(h(g_n))}\leqslant1
                              \qquad(n\in\N).
\end{equation}

Continuity of $h$ and $g_n\to f_0$ imply that $h(g_n)\to h(f_0)$ weakly
in $E$.
Let us put $z_n=h(g_n)-h(f_0)$.  Every weakly convergent sequence in a Banach
space is norm bounded by the uniform boundedness principle, so there is
$M<\infty$ with $\norm{z_n}\leqslant M$ for every $n$.  Now
\begin{align}\label{eq:final-estimate}
 \abs{\mu_n(h(g_n))-\mu_n(h(f_0))}
 &=\abs{\mu_n(z_n)} \notag\\
 &\leqslant \abs{\mu_\infty(z_n)}
       +\norm{\mu_n-\mu_\infty}\,\norm{z_n}
   \longrightarrow0.
\end{align}
The first term tends to $0$ by weak convergence of $z_n$, and the second
by the norm convergence of $(\mu_n)$ and norm boundedness.  Equations
\eqref{eq:margin} and \eqref{eq:final-estimate} give
$\abs{\mu_n(h(g_n))}>1$ for all sufficiently large $n$, contradicting
\eqref{eq:upper-bound}.
\end{proof}

\begin{proof}[Proof of Corollary~B]
If $K$ and $L$ are infinite compact Hausdorff spaces, then $C(L)$ is an
infinite-dimensional Banach space.  Applying Theorem~A with $E=C(L)$
gives the conclusion.
\end{proof}

\subsubsection*{Note added in proof}

Krupski and Marciszewski \cite{KrupskiMarciszewski2026} have subsequently
extended Theorem~A by adapting the methods developed here.  More
precisely, for every infinite Tychonoff space $X$ and every real normed
space $E$, they show that the inverse of any continuous bijection
$h\colon C_p(X)\to E_w$ is discontinuous at every point of $E$.
Consequently, for a Tychonoff space $X$ and a real normed space $E$, one
has $C_p(X)\cong E_w$ if and only if $X$ is finite and
$\dim E=\abs X$.  This settles the broader question of K\k{a}kol,
Leiderman, and Michalak mentioned above.

\subsubsection*{Acknowledgements}

We thank Miko\l{}aj Krupski and Witold Marciszewski for helpful
discussions and for drawing our attention to Sakai's work
\cite{Sakai2006}.

\subsubsection*{Declaration regarding the use of AI}

The work on this paper began before large language models capable of
substantive mathematical reasoning became available.  During the
preparation of a revised version, we used ChatGPT to seek an additional
critical assessment of the proofs.  These checks led to no changes in
the mathematical arguments or results.  The authors take full
responsibility for the content of the manuscript.


\begin{thebibliography}{99}

\bibitem{AlbiacKalton2016}
F.~Albiac and N.~J. Kalton,
\emph{Topics in Banach Space Theory}, second ed.,
Graduate Texts in Mathematics, vol.~233, Springer, Cham, 2016.

\bibitem{Arkhangelskii1992}
A.~V. Arkhangel'ski\u{\i},
\emph{Topological Function Spaces},
Mathematics and its Applications, vol.~78, Kluwer Academic Publishers,
Dordrecht, 1992.

\bibitem{BanakhPlichko2006}
T.~Banakh and A.~Plichko,
\emph{The algebraic dimension of linear metric spaces and Baire
properties of their hyperspaces},
Rev. R. Acad. Cienc. Exactas F\'{\i}s. Nat. Ser. A Mat. RACSAM
\textbf{100} (2006), 31--37.

\bibitem{Beezer2004}
R.~A. Beezer,
\emph{A First Course in Linear Algebra}, version~2.21, 2004,
Section~VM, Theorem~DVM.

\bibitem{ErdosRado1960}
P.~Erd\H{o}s and R.~Rado,
\emph{Intersection theorems for systems of sets},
J. London Math. Soc. \textbf{35} (1960), no.~1, 85--90.

\bibitem{GorakKrupskiMarciszewski2019}
R.~G\'{o}rak, M.~Krupski, and W.~Marciszewski,
\emph{On uniformly continuous maps between function spaces},
Fund. Math. \textbf{246} (2019), 257--274.

\bibitem{KakolLeidermanMichalak2022}
J.~K\k{a}kol, A.~Leiderman, and A.~Michalak,
\emph{A note on Banach spaces $E$ for which $E_w$ is homeomorphic to
$C_p(X)$},
Rev. R. Acad. Cienc. Exactas F\'{\i}s. Nat. Ser. A Mat. RACSAM
\textbf{116} (2022), article~150.

\bibitem{Krupski2013}
M.~Krupski,
\emph{On the $t$-equivalence relation},
Topology Appl. \textbf{160} (2013), 368--373.

\bibitem{Krupski2016}
M.~Krupski,
\emph{On the weak and pointwise topologies in function spaces},
Rev. R. Acad. Cienc. Exactas F\'{\i}s. Nat. Ser. A Mat. RACSAM
\textbf{110} (2016), 557--563.

\bibitem{KrupskiMarciszewski2016}
M.~Krupski and W.~Marciszewski,
\emph{A metrizable $X$ with $C_p(X)$ not homeomorphic to
$C_p(X)\times C_p(X)$},
Israel J. Math. \textbf{214} (2016), no.~1, 245--258.

\bibitem{KrupskiMarciszewski2017}
M.~Krupski and W.~Marciszewski,
\emph{On the weak and pointwise topologies in function spaces II},
J. Math. Anal. Appl. \textbf{452} (2017), 646--658.

\bibitem{KrupskiMarciszewski2026}
M.~Krupski and W.~Marciszewski,
\emph{On condensations of $C_p(X)$ spaces onto Banach spaces with the
weak topology},
preprint, 2026, arXiv:2610.24165.
\url{https://doi.org/10.48550/arXiv.2610.24165}.

\bibitem{KrupskiKucharskiMarciszewski2025}
M.~Krupski, K.~Kucharski, and W.~Marciszewski,
\emph{Characterizing function spaces which have the property~(B) of
Banakh},
Rev. R. Acad. Cienc. Exactas F\'{\i}s. Nat. Ser. A Mat. RACSAM
\textbf{119} (2025), article~75.

\bibitem{Okunev2011}
O.~Okunev,
\emph{A relation between spaces implied by their $t$-equivalence},
Topology Appl. \textbf{158} (2011), 2158--2164.

\bibitem{Sakai2006}
M.~Sakai,
\emph{Two properties of $C_p(X)$ weaker than the Fr\'{e}chet Urysohn
property},
Topology Appl. \textbf{153} (2006), no.~15, 2795--2804.
\url{https://doi.org/10.1016/j.topol.2005.11.012}.

\end{thebibliography}
\end{document}